\documentclass[11pt]{amsart}
\usepackage{amsmath,enumerate,amsfonts,amssymb,amsthm, verbatim}

\newtheorem{theorem}{Theorem}
\newtheorem{lemma}[theorem]{Lemma}
\newtheorem{proposition}[theorem]{Proposition}
\newtheorem{corollary}[theorem]{Corollary}

\theoremstyle{definition}
\newtheorem{definition}[theorem]{Definition}
\newtheorem{example}[theorem]{Example}
\newtheorem{notrems}[theorem]{Notation and Remarks}
\newtheorem{notation}[theorem]{Notation}
\newtheorem{remark}[theorem]{Remark}

\newcommand{\Section}[1]{\section{#1}\setcounter{theorem}{0}}

\newcommand{\<}{\langle}
\renewcommand{\>}{\rangle}
\newcommand{\N}{{\mathbb N}}
\newcommand{\Z}{{\mathbb Z}}
\newcommand{\eps}{{\varepsilon}}
\renewcommand{\phi}{{\varphi}}
\newcommand{\R}{{\mathbb R}}
\newcommand{\C}{{\mathbb C}}
\renewcommand{\H}{{\mathbb H}}
\newcommand{\g}{{\mathfrak g}}
\newcommand{\CC}{{\mathcal{C}}}
\newcommand{\EE}{{\mathcal{E}}}

\newcommand{\scp}{{\<\,\,,\,\>}}
\newcommand{\spann}{{\operatorname{span}}}

\newcommand{\divv}{{\operatorname{div}}}
\newcommand{\grad}{{\operatorname{grad}}}
\newcommand{\res}{{\operatorname{res}}}
\newcommand{\Id}{{\operatorname{Id}}}

\renewcommand{\Re}{{\operatorname{Re}}}
\newcommand{\Imm}{{\operatorname{Im}}}
\newcommand{\symt}{{\operatorname{Sym}^2}}
\newcommand{\symtp}{{\operatorname{Sym}^2_+}}
\newcommand{\dimm}{{\operatorname{dim\,}}}

\newcommand{\diag}{{\operatorname{diag}}}
\newcommand{\inv}{^{-1}}
\newcommand{\Irr}{{\operatorname{Irr}}}
\newcommand{\End}{{\operatorname{\textsl{End}}}}
\newcommand{\GL}{{\operatorname{\textsl{GL}}}}
\newcommand{\SU}{{\operatorname{\textsl{SU}}}}
\newcommand{\U}{{\operatorname{\textsl{U}}}}
\newcommand{\su}{{\operatorname{\mathfrak{su}}}}
\newcommand{\SO}{{\operatorname{\textsl{SO}}}}
\newcommand{\Spin}{{\operatorname{\textsl{Spin}}}}

\newcommand{\restr}[1]{\lower0.4ex\hbox{$|$}\lower0.7ex
  \hbox{$\scriptstyle{#1}$}}
\newcommand{\diffzero}[1]{{\frac d{d{#1}}\lower1ex\hbox{$\big\vert_{\raise1\jot
         \hbox{${\tsize {#1}=0}$}}$}\,}}
\newcommand{\tdiffzero}[1]{{\frac d{d{#1}}\lower0.2ex\hbox{\restr{{#1}=0}}\,}}
\newcommand{\diffzerosec}[1]{{\frac{d^2}{d{#1}^2}\lower1ex\hbox{$\big\vert_{\raise
        1\jot\hbox{${\tsize {#1}=0}$}}$}\,}}
\newcommand{\tdiffzerosec}[1]{{\frac{d^2}{d{#1}^2}\lower0.2ex\hbox{\restr{{#1}=0}}\,}}

\begin{document}

\title[Generic irreducibility of Laplace eigenspaces]{Generic irreducibilty of
  Laplace eigenspaces on certain compact Lie groups}
\author{Dorothee Schueth}
\address{Institut f\"ur Mathematik, Humboldt-Universit\"at zu
Berlin, D-10099 Berlin, Germany}
\email{schueth@math.hu-berlin.de}


\thanks{The author was partially supported by by DFG Sonderforschungsbereich~647.}

\begin{abstract}
If $G$ is a compact Lie group endowed with a left invariant metric~$g$,
then $G$ acts via pullback by isometries on each eigenspace of the associated
Laplace operator~$\Delta_g$.
We establish algebraic criteria for the existence of left invariant metrics~$g$ on~$G$
such that each eigenspace of~$\Delta_g$, regarded as the real vector space of
the corresponding real eigenfunctions, is irreducible under the action of~$G$.
We prove that generic left invariant metrics on the Lie groups
$G=\SU(2)\times\ldots\times\SU(2)\times T$, where $T$ is a (possibly trivial) torus,
have the property just described. The same holds for quotients of such groups~$G$
by discrete central subgroups. In particular, it also holds for $\SO(3)$, $\U(2)$, $\SO(4)$.
\end{abstract}

\maketitle

\Section{Introduction}
\label{sec:intro}

\noindent
Let $(M,g)$ be a closed connected Riemannian manifold. The eigenvalue spectrum (with multiplicities)
of the associated Laplace operator $\Delta_g=-\divv_g\grad_g$ acting on smooth functions forms
a discrete series $0=\lambda_0<\lambda_1\le\lambda_2\le\ldots\to\infty$.

A classical result by K.~Uhlenbeck~\cite{U} says that for a generic Riemannian metric~$g$ on~$M$,
all eigenvalues of~$\Delta_g$ are
simple (i.e., have multiplicity
one). At the other extreme, if~$g$ is a homogeneous metric, i.e. the group of isometries
acts transitively on~$M$, then every nonzero eigenvalue is
necessarily multiple; this is a consequence of that each eigenspace is
invariant under pullback by isometries.
An interesting question in this context, raised by V.~Guillemin, is whether on a
a compact Lie group~$G$ there always exists a left invariant metric~$g$ such
that $G$ acts irreducibly on each eigenspace
of~$\Delta_g$\,. In other words, the question is whether for metrics~$g$ which are ``generic''
\emph{within} the set of left invariant Riemannian metrics on~$G$,
the eigenvalues of~$\Delta_g$ have no higher multiplicities than
necessitated by the prescribed symmetries.

For left invariant metrics on~$G$, the associated Laplacian can be expressed via the right regular
representation of~$G$ on $C^\infty(G,\C)$.
Note that the case of biinvariant metrics on simple compact Lie groups represents the
most ``nongeneric'' case here: For such metrics the Laplacian corresponds to a scalar
multiple of the
Casimir operator, and thus has only one eigenvalue on each isotypical component in
$C^\infty(G,\C)$; since the isotypical components are not irreducible (by the
Peter-Weyl theorem), the eigenspaces are certainly not irreducible for a biinvariant metric.

Using the explicit description of the isotypical components of the right regular representation
from the Peter-Weyl theorem,
one quickly arrives at a tentative reformulation for irreducibility of the eigenspaces of~$\Delta_g$
for a given left invariant metric~$g$:
Roughly speaking, for each irreducible representation $\rho_V : G\to\GL(V)$ the eigenvalues
of the operator $\Delta_g^V:=-\sum_{k=1}^n((\rho_V)_*(Y_k))^2$ (where $\{Y_1,\ldots,Y_n\}$ is a 
orthonormal basis of $\g=T_e G$) should be simple, and two nonisomorphic
representations should not share a common eigenvalue (see Remark~\ref{rem:lap}(ii)). 
However, these properties
can never be satisfied if $G$ admits irreducible representations of so-called quaternionic
type (on which all eigenvalues will have even multiplicity) or of complex type (on which
the eigenvalues will be the same as on the -- nonisomorphic -- dual representation);
see Remark~\ref{rem:comquat}.

Fortunately, it turns out that when one considers real-valued eigenfunctions,
then these complications no longer form an obstacle to irreducibility of the eigenspaces.
Rather, the latter then becomes equivalent
to the following three conditions being jointly satisfied: Simple eigenvalues of $\Delta_g^V$
on each irreducible representation~$V$ of real or complex type; eigenvalues of multiplicity
precisely two on each irreducible representation of quaternionic type; no common eigenvalues
of $\Delta_g^V$, $\Delta_g^W$ whenever $V,V^*\not\cong W$ (Corollary~\ref{cor:irr}).

Expressing these conditions in terms of certain resultants or discriminants
of the characteristic polynomials of
the operators $\Delta_g^V$
(or of their derivatives) being nonzero leads to the description of the set of left invariant
metrics with the desired property as the intersection of the complements of the zero sets
of countably many polynomials on $\symtp(\g):=\{Y_1^2+\ldots Y_n^2\mid \{Y_1,\ldots,Y_n\}
\text{ a basis of }\g\}$. Since $\symtp(\g)$ is an open subset of $\symt(\g)\subset
\g\otimes\g$, it simplifies the discussion to regard these polynomials as defined on
all of $\symt(\g)$. Summarizing, existence of a left invariant metric with irreducible real eigenspaces
is equivalent to the condition that none of certain countably many polynomials on $\symt(\g)$
is the zero polynomial; see Proposition~\ref{prop:abc}. In that case, the intersection
of the complements of the zero sets will not only be nonempty, but even residual.

We apply this general description to prove that the Lie group $\SU(2)$ and also
products of the form $\SU(2)\times\ldots\SU(2)\times T$, where $T$ is a torus,
do have the property that generic left invariant metrics on these groups have
irreducible real eigenspaces; see Theorems~\ref{thm:su2} and~\ref{thm:spinmult}.
For $\SU(2)$, the key of the proof consists in showing
that for those of its irreducible representations~$V$ which are of real type, the
eigenvalues of $\Delta_g^V$ are generically simple; the other conditions of Proposition~\ref{prop:abc}
are almost obvious here. For products $\SU(2)\times\SU(2)$, the main difficulty is
showing generic simplicity of eigenvalues on irreducible representations of real type
of the form
$V\otimes W$, where $V$ and~$W$ are irreducible representations of $\SU(2)$ of
quaternionic type; see Remark~\ref{rem:nontriv} and Lemma~\ref{lem:pairs}(ii).

Finally, we observe that if a compact Lie group~$G$ satisfies the conditions
of Proposition~\ref{prop:abc}, then so do its quotients by discrete central subgroups;
see Lemma~\ref{lem:quot}.
Therefore the result extends, for example, to $\SO(3)$, $\U(2)$, and $\SO(4)$.

Note that all of the operators $\Delta_g^V$ are hermitian with respect to a $G$-invariant
hermitian inner product on~$V$. It is well-known that for analytic $1$-parameter families (although not
for analytic multiparameter families) of such operators, the eigenvalues are analytic functions of
the parameter. This fact and methods from perturbation theory as in~\cite{K} might
be useful when examining the problem for other groups. However, the fact that operators of the form
$\Delta_g^V$ lie in a quite small subset of all hermitian operators on~$V$ constitutes a major
difficulty. Our proofs for $\SU(2)$ and $\SU(2)\times\ldots
\times\SU(2)\times T^n$ do actually not use any general perturbation theoretic arguments.

\medskip
This paper is organized as follows:

In Section~\ref{sec:prelims}, we state some basic facts about complex
irreducible representations of compact Lie groups~$G$ and describe how
the Laplace operator~$\Delta_g$ associated with a left invariant metric~$g$
on~$G$ acts on the isotypical components of the right regular
representation of~$G$ on~$C^\infty(G,\C)$.

In Section~\ref{sec:real}, we establish representation theoretic criteria
for the existence of a left
invariant metric~$g$ on~$G$ such that each real eigenspace of~$\Delta_g$
is irreducible (Proposition~\ref{prop:abc}). We observe that in case
of existence, generic left invariant metrics on~$G$ have the same property.
As an illustration, we discuss the case $G=T^n$ (where the said property
of generic left invariant metrics is well-known).

In Section~\ref{sec:groups}, we first prove that the Laplace operators~$\Delta_g$
associated with generic left invariant
metrics on
$\SU(2)$ have irreducible real eigenspaces (Theorem~\ref{thm:su2}). After examining
which of the criteria of Proposition~\ref{prop:abc} are, resp.~are not, easily seen to be
inherited by products of two Lie groups from their factors,
we extend the above result to products
of the form $\SU(2)\times\ldots\times\SU(2)\times T^n$ (Theorem~\ref{thm:spinmult})
and, as a corollary, to quotients of these groups
by discrete central subgroups.

\medskip
The author would like to thank to thank
Carolyn S.~Gordon, David L.~Webb, and Victor Guillemin for inspiring discussions,
and the latter especially for first drawing our attention to the topic.
Moreover, she would like to thank Dartmouth College for its hospitality during
a stay where this research was initiated.

\Section{Preliminaries}
\label{sec:prelims}

\begin{notation}\
\label{not:regular}
\begin{itemize}
\item[(i)]
Throughout the paper, we let $G$ be an $n$-dimensional compact Lie group with Lie algebra~$\g$.
By $\ell_x:G\to G$ (resp.~$r_x:G\to G$) we denote
left (resp.~right) multiplication by $x\in G$.
By $L$ (resp.~$R$) we denote the left regular (resp.~right regular) unitary representation of~$G$ on $L^2(G,\C)$,
given by
\begin{equation*}
R(x)f := f\circ r_x: y\mapsto f(yx)\text{\ \ and \ }L(x)f :=f\circ \ell_{x\inv}:y\mapsto f(x\inv y)
\end{equation*}
for $f\in L^2(G,\C)$. Of course, the two regular representations are isomorphic to each other
via $f\mapsto f\circ\operatorname{inv}$.
\item[(ii)]
If $\rho$ is a representation of~$G$ on some real or complex vector space $V$, then $V$ together with
the action of~$G$ by~$\rho$ is called a $G$-module.
We choose sets $\Irr(G,\C)$ and $\Irr(G,\R)$ of representatives of
isomorphism classes of irreducible real, resp.~complex, $G$-modules.
Since $G$ is compact, all irreducible $G$-modules are finite dimensional.
\item[(iii)]
For an irreducible complex $G$-module~$V$, denote by $I(V)\subset L^2(G,\C)$ the $V$-isotypical
component with respect to the right regular
representation~$R$ on~$L^2(G,\C)$.
\item[(iv)]
A complex irreducible $G$-module~$V$ is called \emph{of real type} (resp.~\emph{of quaternionic type}) if there exists a
conjugate linear $G$-map $J:V\to V$ such that $J^2=\Id$ (resp.~$J^2=-\Id$); $V$~is called \emph{of
complex type} if it is of neither real nor quaternionic type.
\end{itemize}
\end{notation}

\begin{lemma}[see, e.g.,~\cite{BtD}, section II.6]\
\label{lem:irrc}
$\Irr(G,\C)$ is the disjoint union of
$\Irr(G,\C)_\R$, $\Irr(G,\C)_\C$, and $\Irr(G,\C)_\H$, where these denote the subsets consisting
of those elements which are of real, resp.~complex, resp.~quaternionic type.
For $V\in\Irr(G,\C)$ these properties can be characterized as follows:
\begin{itemize}
\item[(i)]
\;$V\in\Irr(G,\C)_\R$ $\Longleftrightarrow$ $V\cong V^*$ and $V\cong U\otimes\C$ for some $U\in\Irr(G,\R)$.
\item[(ii)]
\;$V\in\Irr(G,\C)_\C$ $\Longleftrightarrow$ $V\not\cong V^*$ and $V\oplus V^*\cong U\otimes \C$ for some $U\in\Irr(G,\R)$.
\item[(iii)]
\;$V\in\Irr(G,\C)_\H$ $\Longleftrightarrow$ $V\cong V^*$ and $V\oplus V\cong U\otimes\C$ for some $U\in\Irr(G,\R)$.
\end{itemize}
\end{lemma}

\begin{remark}[see, e.g.,~\cite{BtD}, sections III.1--III.3]\
\label{rem:pw}
For $V\in\Irr(G,\C)$ let $\rho_V$ denote the representation of~$G$ on~$V$.
By the Peter-Weyl Theorem, 
each of the isotypical components $I(V)$ of the right regular representation~$R$ of~$G$ on $L^2(G,\C)$
is contained in $C^\infty(G,\C)$,
and $\bigoplus_{V\in\Irr(G,\C)} I(V)$ is dense in $L^2(G,\C)$.
Moreover, each $I(V)$ is invariant under both $R$ and $L$,
and there is a vector space isomorphism
\begin{equation}
\label{eq:phiv}
\phi_V:V^*\otimes V\to I(V)\text{\ \ given by }\lambda\otimes v\mapsto\lambda(\rho_V(\,.\,)v).
\end{equation}\pagebreak

Thus, on $V^*\otimes V$ one has
\begin{equation}
\label{eq:vv}
\Id\otimes\rho_V(x)=\phi_V\inv\circ R(x)\circ \phi_V\text{ and }
(\rho_V(x)\inv)^*\otimes\Id=\phi_V\inv\circ L(x)\circ \phi_V
\end{equation}
for all $x\in G$.
In particular, $I(V)$ is not only the $V$-isotypical component with respect to~$R$, but also the
$V^*$-isotypical component with respect to $L$.
\end{remark}

\begin{remark}
\label{rem:lap}
Let $g$ be a left invariant Riemannian metric on~$G$.

(i)
The Laplace operator~$\Delta_g$ associated with~$g$
acts on $C^\infty(G,\C)$ by
\begin{equation*}
\Delta_g f=-\textstyle\sum_{k=1}^n \tdiffzerosec{t} f(\,.\,e^{tY_k})=-\textstyle\sum_{k=1}^n (R_*(Y_k))^2 f,
\end{equation*}
where $\{Y_1,\ldots, Y_n\}$ is a $g$-orthonormal basis of~$\g$. This well-known formula follows
from unimodularity of~$G$ and the fact that for each $y\in G$, the initial velocity vectors of the curves
$t\mapsto ye^{tY_k}$ constitute a $g$-orthonormal basis at~$y$.

(ii)
For each $V\in\Irr(G,\C)$, the isotypical component $I(V)$ is invariant under $\Delta_g$ by~(i) and
Remark~\ref{rem:pw}. 
More precisely, by~\eqref{eq:vv}:
\begin{equation}
\label{eq:lapv}
\phi_V\inv\circ\Delta_g\restr{I(V)}\circ\phi_V=\Id\otimes\Delta_g^V\text{\ \ on }V^*\otimes V,
\end{equation}
where
$$\Delta_g^V:=-\sum_{k=1}^n ((\rho_V)_*(Y_k))^2:V\to V
$$
and $\{Y_1,\ldots,Y_n\}$ is a $g$-orthonormal basis.
In particular, each eigenvalue of the restriction of~$\Delta_g$ to the complex vector space~$I(V)$
has multiplicity at least~$\dimm V^*=\dimm V$,
and irreducibility of the eigenspaces of $\Delta_g\restr{I(V)}$ w.r.t.~the left regular representation~$L$
is equivalent to these multiplicities being precisely~$\dimm V$ and the eigenvalues of $\Delta_g^V$ being simple.

(iii)
In the context of~(ii), note that $(\rho_{V^*})_*(X)\lambda=-\lambda\circ(\rho_V)_*X$ for $X\in\g$, $\lambda\in V^*$,
hence
$(\Delta_g^{V^*})(\lambda)=\lambda\circ\Delta_g^V$. Thus, the dual basis of an eigenbasis
of~$\Delta_g^V$ is always an eigenbasis of $\Delta_g^{V^*}$ with the same eigenvalues.
\end{remark}

\begin{remark}
\label{rem:comquat}
(i)
If $V\in\Irr(G,\C)$ is of complex type, i.e., $V\not\cong V^*$, then the two isotypical
components $I(V)$ and $I(V^*)$ do not coincide. However, by Remark~\ref{rem:lap}(ii) and~(iii), the eigenvalues
of $\Delta_g$ on $I(V)$ are the same as those on $I(V^*)$, for any left invariant metric~$g$ on~$G$.
In particular, the corresponding eigenspaces are not irreducible w.r.t.~$L$.

(ii)
If $V\in\Irr(G,\C)$ is of real or quaternionic type, i.e., $V\cong V^*$, then $I(V)=I(V^*)$.
If $V$ is of quaternionic type and $J:V\to V$ is as
in~\ref{not:regular}(iv) with $J^2=-\Id$, then each eigenspace of~$\Delta_g^V$ invariant under~$J$.
Since the eigenvalues of~$\Delta_g^V$ are real, this invariance together with $J^2=-\Id$
implies that each eigenspace is of even dimension.
In particular, it follows by~\ref{rem:lap}(ii) that the eigenspaces of
$\Delta_g\restr{I(V)}$ itself are never irreducible w.r.t.~$L$ if $V$ is of quaternionic type.
\end{remark}

The situation just described changes if one shifts attention to irreducibility of \emph{real}
eigenspaces, as we will see in the following section.

\Section{Irreducibility conditions for real eigenspaces}
\label{sec:real}

\begin{lemma}
\label{lem:real}
For $V\in\Irr(G,\C)$ define $\CC_V\subset C^\infty(G,\C)$ by
$$
\CC_V:=\begin{cases} I(V),& V\text{ of real or quaternionic type},\\
                     I(V)\oplus I(V^*),& V\text{ of complex type}
\end{cases}
$$	
and
$$\EE_V:=\CC_V\cap C^\infty(G,\R).
$$
Obviously, $\EE_V$ is invariant under $L$ and~$R$; moreover:
\begin{itemize}
\item[(i)]
The complexification of $\EE_V$ is $\CC_V$.
\item[(ii)]
For any left invariant metric~$g$ on~$G$,
the following conditions are equivalent:
\begin{itemize}
\item[$\bullet$] Each eigenspace of $\Delta_g\restr{\EE_V}$ is an irreducible real $G$-module with respect to~$L$,
\item[$\bullet$] each eigenvalue of $\Delta_g^V$ has multiplicity
$$\begin{cases}1,& V\text{ of real or complex type},\\
           2,& V\text{ of quaternionic type}.
           \end{cases}
$$
\end{itemize}
\end{itemize}
\end{lemma}

\begin{proof}
(i)
Complex conjugation $L^2(G,\C)\ni f\mapsto
\bar f\in L^2(G,\C)$ maps $I(V)$ to $I(\bar V)$, and
$I(\bar V)=I(V^*)$ since $\bar V$ and~$V^*$ are isomorphic. 
In particular, $I(V)+I(V^*)$
is invariant under the projections to the real and 
imaginary parts of functions. The statement now follows
by recalling that $I(V)=I(V^*)$ if $V$ is of real or 
quaternionic type.

(ii)
Let~$\mu$ be an eigenvalue of $\Delta_g^V$ with
multiplicity~$m$, and let $V^\mu\subset V$, $(V^*)^\mu
\subset V^*$, $\CC_V^\mu\subset \CC_V$, and
$\EE_V^\mu\subset\EE_V$ denote the corresponding
eigenspaces of $\Delta_g^V$, $\Delta_g^{V^*}$,
$\Delta_g\restr{\CC_V}$, $\Delta_g\restr{\EE_V}$,
respectively. Recall from Remark~\ref{rem:lap}(iii) that
$\dimm (V^*)^\mu=\dimm V^\mu=m$.
By \eqref{eq:vv} and~\eqref{eq:lapv}, $\CC_V^\mu$ is
invariant under~$L$ and, as a complex $G$-module, satisfies
$$
\CC_V^\mu\cong
\begin{cases}
V^*\otimes V^\mu\cong (V^*)^{\oplus m}\cong V^{\oplus m},&V \text{ of real or quaternionic type},\\
(V^*\otimes V^\mu)\oplus(V\otimes (V^*)^\mu)
\cong (V^*\oplus V)^{\oplus m},&V \text{ of complex type}.\end{cases}
$$
Lemma~\ref{lem:irrc} now implies that in each case there exists some $U\in\Irr(G,\R)$ such that
$$
\CC_V^\mu\cong
\begin{cases}
(U\otimes\C)^{\oplus m}\cong U^{\oplus m}\otimes\C,&
  V\text{ of real or complex type},\\
(U\otimes\C)^{\oplus m/2}\cong U^{\oplus m/2}\otimes\C,&
  V\text{ of quaternionic type}.
\end{cases}
$$
(Recall from Remark~\ref{rem:comquat}(ii) that $m$ is 
necessarily even if $V$ is of quaternionic type.)

On the other hand, we clearly have $\CC_V^\mu
=\EE_V^\mu\otimes\C$ by~(i) and since $\EE_V$ is invariant 
under~$\Delta_g$.
Regarded as a \emph{real} $G$-module, $\CC_V^\mu$ is thus 
isomorphic to $\EE_V^\mu\oplus\EE_V^\mu$ on the one hand, 
and to $U^{\oplus m}\oplus U^{\oplus m}$,
resp.~$U^{\oplus m/2}\oplus U^{\oplus m/2}$, on the other 
hand. Since $U$ is irreducible, we conclude
$$
\EE_V^\mu\cong
\begin{cases}
U^{\oplus m},&V\text{ of real or complex type},\\
U^{\oplus m/2},&V\text {of quaternionic type}.
\end{cases}
$$
In particular, $\EE_V^\mu$ is irreducible if and only if
$m=1$ for $V$ of real or complex type, resp.~$m=2$ for $V$ 
of quaternionic type.
\end{proof}


\begin{definition}
Given a left invariant metric~$g$ on~$G$, we say that
$\Delta_g$ \emph{has irreducible real eigenspaces} if each eigenspace of
$\Delta_g\restr{C^\infty(G,\R)}$ is irreducible with 
respect to the 
action~$L$ of~$G$.
\end{definition}

{}From Lemma~\ref{lem:real} and Remark~\ref{rem:lap} we 
immediately obtain:

\begin{corollary}
\label{cor:irr}
Let $g$ be a left invariant metric on~$G$. Then
$\Delta_g$ has irreducible real eigenspaces if and only if each 
of the following conditions is satisfied:
\begin{itemize}
\item[(i)]
For any pair $V,W\in\Irr(G,\C)$ with $V\not\cong W$ and
$V^*\not\cong W$, $\Delta_g^V$ and $\Delta_g^W$ have no 
common eigenvalues.
\item[(ii)]
For each $V\in\Irr(G,\C)$ of real or complex type, all 
eigenvalues of $\Delta_g^V$ are simple.
\item[(iii)]
For each $V\in\Irr(G,\C)$ of quaternionic type,
all eigenvalues of $\Delta_g^V$ are of multiplicity two.
\end{itemize}
\end{corollary}\pagebreak

\begin{notrems}\
\label{notrems:abc}
\begin{itemize}
\item[(i)]
Let $\symt(\g):=\spann_\R\{Y\cdot Z:=\frac12(Y\otimes Z+Z\otimes Y)\mid Y,Z\in\g\}$
be the second symmetric tensor power of~$\g$.
We also write $Y^2$ for $Y\cdot Y$.
\item[(ii)]
By $\symtp(\g)$ we denote the subset $\{Y_1^2+\ldots+Y_n^2\mid\{Y_1,\ldots,Y_n\}\text{ is a basis of }\g\}$.
Note that $\symtp(\g)$ is open in~$\symt(\g)$. (In fact, if we identify $\symt(\g)$ with the
space of real symmetric $n\times n$-matrices by fixing a basis of~$\g$ and the corresponding
canonical basis of~$\symt(g)$, then
$\symtp(\g)$ corresponds to the subset of positive definite matrices.)
\item[(iii)]
For $V\in\Irr(G,\C)$, we define the linear map $D_V:\symt(\g)\to\End(V)$ by
$$D_V(Y\cdot Z):=-\tfrac12\bigl((\rho_V)_*(Y)\circ(\rho_V)_*(Z)
  +(\rho_V)_*(Z)\circ(\rho_V)_*(Y)\bigr)
$$
for $Y,Z\in\g$, and by linear extension.
Note that any endomorphism in the image of~$D_V$ is
diagonizable and has only real eigenvalues because
it is a hermitian map with respect to any $G$-invariant
hermitian inner product on~$V$.
\item[(iv)] 
If $\{Y_1,\ldots,Y_n\}$ is a basis of~$\g$ then
$$D_V(Y_1^2+\ldots+Y_n^2)=-\textstyle\sum_{k=1}^n((\rho_V)_*(Y_k))^2=\Delta_g^V,
$$
where $g$ is the metric 
with orthonormal basis $\{Y_1,\ldots,Y_n\}$.
\end{itemize}
\end{notrems}

\begin{definition}
\label{def:abc}
\begin{itemize}
\item[(i)]
Let 
$$p_V:=\det(D_V(\,.\,)-X\cdot\Id):\symt(\g)\to\C[X]
$$
be the map sending $s\in\symt(\g)$ to the characteristic 
polynomial of $D_V(s)$.
\item[(ii)]
By
$$\res:\C[X]\times\C[X]\to\C
$$
we denote the resultant; see, e.g.,~\cite{GKZ}.
For two polynomials $p,q\in\C[X]$ the number $\res(p,q)$ is 
given by a 
certain polynomial in the coefficients of $p$ and~$q$, and
$$
\res(p,q)\ne 0\Longleftrightarrow p\text{ and }q\text{ have
no common zeros}.
$$
For a polynomial $p$ and its formal derivative $p'$, $\res(p,p')$
is the discriminant of $p$ (up to some nonzero scalar factor)
and vanishes if and only if $p$ has a zero of multiplicity at least two.
\item[(iii)]
For $V,W\in\Irr(G,\C)$ we define the following
$\C$-valued polynomials on~$\symt(\g)$:
\begin{equation*}
\begin{split}
a_{V,W}&:=\res\circ(p_V,p_W):\symt(\g)\to\C,\\
b_V&:=\res\circ(p_V,p'_V):\symt(\g)\to\C,\\
c_V&:=\res\circ(p_V,p''_V):\symt(\g)\to\C,
\end{split}
\end{equation*}
where $p'_V, p''_V\in\C[X]$ denote the formal derivatives
of $p_V\in\C[X]$ with respect to the variable~$X$.
\end{itemize}
\end{definition}

\begin{remark}
\label{rem:abc}
(i)
Since $\symtp(\g)$ is open in~$\symt(\g)$, the polynomial $a_{V,W}$ on~$\symt(\g)$ is not identically zero
if and only if it is not identically zero on $\symtp(\g)$.
This is the case if and only if there exists $s\in\symtp(\g)$ such that $D_V(s)$
and~$D_W(s)$ have no common zeros. By the definition of~$\symtp(\g)$ and by~\ref{notrems:abc}(iv),
this is equivalent to the existence of a left invariant metric~$g$ on~$G$
such that $\Delta_g^V$ and $\Delta_g^W$ have no common 
eigenvalues.

(ii)
Similarly, $b_V\ne0$ is equivalent to the existence of a left invariant
metric~$g$ on~$G$ such that all eigenvalues of $\Delta_g^V$ have multipicity one. 

(iii)
If $V$ is of quaternionic type then all eigenvalues
of~$\Delta_g$ have at least multiplicity two by
Remark~\ref{rem:comquat}(ii).
In this case, $c_V\ne0$ is equivalent to the existence of a left invariant
metric~$g$ on~$G$ such that all eigenvalues of~$\Delta_g^V$ are of
multiplicity exactly two.
\end{remark}

\begin{proposition}
\label{prop:abc}
Existence of a left invariant metric~$g$ on~$G$ such that $\Delta_g$ has
irreducible real eigenspaces is equivalent to the following conditions 
being jointly satisfied:
\begin{itemize}
\item[(a)]
$a_{V,W}\ne0$ for any pair $V,W\in\Irr(G,\C)$ with
$V\not\cong W$ and $V^*\not\cong W$,
\item[(b)]
$b_V\ne0$ for each $V\in\Irr(G,\C)$ of real or complex type,
\item[(c)]
$c_V\ne0$ for each $V\in\Irr(G,\C)$ of quaternionic type.
\end{itemize}
In this case, the orthonormal bases for left invariant
metrics~$g$ with the property that $\Delta_g$ has irreducible real eigenspaces constitute
a residual set in~$\g^{\oplus n}=\g^{\oplus\dim\g}$.
\end{proposition}

\begin{proof}
That the conditions are necessary is obvious from 
Corollary~\ref{cor:irr} and Remark~\ref{rem:abc}. 
Conversely, assume that (a), (b), (c) are satisfied.
Write
$$F:\g^{\oplus n}\ni(Y_1,\ldots,Y_n)\mapsto Y_1^2+\ldots +Y_n^2\in\symtp(\g)\subset\symt(\g).
$$
Since the image of~ $F$ contains $\symtp(\g)$ which is open in $\symt(\g)$ (see~\ref{notrems:abc}(ii)), the polynomials
$\tilde a_{V,W}:=a_{V,W}\circ F$, $\tilde b_V:=b_V\circ F$, and $\tilde c_V:=c_V\circ F$ are again nontrivial
under the respective assumptions on $V$ and $W$.
Consider the subset
$$
\mathcal{N}:=\bigcup_{V,W\in\Irr(G,\C);\,V,V^*\not\cong W}\tilde a_{V,W}\inv(0)
\cup\bigcup_{V\in\Irr(G,\C)_\R\cup\Irr(G,\C)_\C}\tilde b_V\inv(0)
\cup\bigcup_{V\in\Irr(G,\C)_\H}\tilde c_V\inv(0)\subset\g^{\oplus n}.
$$
Then $\mathcal{N}$ is the union of the zero sets of countably many
nonzero polynomials. Thus, 
$\g^{\oplus n}\setminus\mathcal{N}$ is a residual set (i.e., an intersection
of countably many sets with dense interiors). Now let
$$
\mathcal{B}:=\{(Y_1,\ldots,Y_n)\in\g^{\oplus n}\setminus\mathcal{N}\mid \{Y_1,\ldots,Y_n\}
\text{ is linearly independent}\}.
$$
Then $\mathcal{B}$ is still residual in~$\g^{\oplus n}$, and for any $b\in\mathcal{B}$
the Laplace operator~$\Delta_g$ associated with
the left invariant metric~$g$ on~$G$ with orthonormal
basis~$b$ has irreducible real eigenspaces by Corollary~\ref{cor:irr}
and Remark~\ref{rem:abc}.
\end{proof}

\begin{example}
\label{ex:torus}
Let $G:=T^n=\R^n/\Z^n$. It is well-known that for generic
left invariant metrics~$g$ on~$T^n$, the Laplace operator~$\Delta_g$
has irreducible real eigenspaces.
In fact, let $\scp$ be a euclidean inner product on~$\R^n$
and $g$ be the corresponding left invariant metric induced
on~$T^n$. Let $\Lambda:=(\Z^n)^*\subset(\R^n)^*$.
For $\lambda\in\Lambda$, we denote the induced function
on~$T^n$ again by~$\lambda$.
The character
$\chi_\lambda:T^n\ni x\mapsto\exp(2\pi i\lambda(x))\in\C$ 
is a complex eigenfunction
of~$\Delta_g$ with
eigenvalue $\mu_\lambda:=4\pi\|\lambda\|^2$, where
$\|\,.\,\|$ denotes
the norm induced on $(\R^n)^*$ by~$\scp$.
For generic~$\scp$, one has $\mu_\lambda=\mu_{\lambda'}$
if and only if $\lambda'=\pm\lambda$. In this case, the
corresponding real eigenspace $\EE^{\mu_\lambda}$
is two-dimensional if $\lambda\ne0$ (otherwise,
one-dimensional) and is spanned by
$\Re(\chi_\lambda)=\cos2\pi\lambda(\,.\,)$ and
$\Imm(\chi_\lambda)=\sin2\pi\lambda(\,.\,)$. Obviously, 
$\EE^{\mu_\lambda}$ is then
irreducible under the action of~$T^n$.

Let us verify this property of~$G=T^n$ in the framework of
Proposition~\ref{prop:abc}:
We have $\Irr(G,\C)=\{V_\lambda\mid\lambda\in\Lambda\}$,
where each $V_\lambda$ is one-dimensional and
$\rho_{V_\lambda}=\chi_\lambda(\,.\,)\Id$. The trivial
representation~$V_0$ is of real type; for $\lambda\ne0$,
$V_\lambda$ is of complex type since $(V_\lambda)^*
=V_{-\lambda}\not\cong V_\lambda$. Writing
$\mu_\lambda(Y\cdot Z):=4\pi^2\lambda(Y)\lambda(Z)$ for
$Y\cdot Z\in\symt(\g)=\symt(\R^n)$, we have $D_{V_\lambda}(Y\cdot Z)
=-(\chi_{\lambda})_*(Y)(\chi_{\lambda})_*(Z)=\mu_\lambda(Y\cdot Z)$.
After extending $\mu_\lambda$ linearly,
we get $D_{V_\lambda}(s) = \mu_\lambda(s)\Id$ and
$$
p_{V_\lambda}(s)=\mu_\lambda(s)-X\in\C[X]
$$
for all $s\in\symt(\R^n)$.
If $V_\lambda,V_\lambda^*\not\cong V_{\lambda'}$
then $\lambda'\ne\pm\lambda$. In this case, choose any
$Y\in\g=\R^n$ with $\lambda'(Y)\ne\pm\lambda(Y)$. Then
$\mu_\lambda(Y^2)\ne\mu_{\lambda'}(Y^2)$, hence
$a_{V_\lambda,V_{\lambda'}}(Y^2)\ne0$.
Moreover, $(p_{V_\lambda}(s))'$ is the constant
polynomial~$-1$, so $b_{V_\lambda}(s)\ne0$ for
every $s\in\symt(\R^n)$. (Of course, this corresponds here to the
trivial fact that the operator $D_{V_\lambda}(s)$ on
the one-dimensional space $V_\lambda$ can only have
a simple eigenvalue.) So conditions (a) and~(b)
of Proposition~\ref{prop:abc} are satisfied;
condition~(c) is void since
none of the $V_\lambda$ is of quaternionic type.
\end{example}

\Section{$\SU(2)$, products, and quotients}
\label{sec:groups}

\begin{theorem}
\label{thm:su2}
The compact Lie group $\SU(2)$ satisfies the conditions of
Proposition~\ref{prop:abc}. In particular, for generic
left invariant metrics~$g$ on~$\SU(2)$, $\Delta_g$ has
irreducible real eigenspaces.
\end{theorem}

\begin{proof}
As is well-known,
the isomorphism classes of irreducible complex
representations of~$G:=\SU(2)$ are parametrized
by $m\in\N_0$, and the corresponding
$(m+1)$-dimensional representation~$V_m$
can be viewed as the space
of homogeneous complex polynomials of degree~$m$ in two
variables: Let
$$
v_{m,\ell}:=z_1^{m-\ell}z_2^\ell\in\C[z_1,z_2],
$$
and $V_m:=\spann\{v_{m,0},\ldots,v_{m,m}\}$.
For $x\in\SU(2)$ define $\rho_{V_m}(x):
V_m\ni v\mapsto v\circ x\in V_m$, where $v\in V_m$
is viewed as a function on $\C^2$
and $\SU(2)$ acts on~$\C^2$ on the right
via $x=\left(\begin{smallmatrix}a&b\\c&d\end{smallmatrix}
\right):(u,v)\mapsto(au+cv,bu+dv)$.
The irreducible representation~$V_m$ of~$\SU(2)$
is of real type if $m$ is odd, and of quaternionic
type if $m$ is even
(see, e.g., \cite{BtD}, section VI.4--VI.5).

Consider the basis $\{H,A,B\}$ of $\g=\su(2)$ with
$$
H:=\left(\begin{smallmatrix}i&0\\0&-i
\end{smallmatrix}\right),
A:=\left(\begin{smallmatrix}0&i\\i&0
\end{smallmatrix}\right),
B:=\left(\begin{smallmatrix}0&-1\\1&0
\end{smallmatrix}\right).
$$
Note that $D_{V_m}(H^2+A^2+B^2)$ is,
up to some scalar multiple,
the Casimir operator on~$V_m$. More precisely, 
one easily sees
$$
D_{V_m}(H^2+A^2+B^2)=m(m+2)\Id_{V_m},
$$
so $m(m+2)$ is the only zero of $p_{V_m}(H^2+A^2+B^2)$.
In particular, $a_{V_m,V_{m'}}(H^2+A^2+B^2)\ne0$ whenever
$m\ne m'$. This shows condition~(a) of
Proposition~\ref{prop:abc}.

Condition~(c) is similarly easy to check:
The matrix corresponding to $(\rho_{V_m})_*(H)$
with respect to the basis $\{v_{m,0},\ldots,v_{m,m}\}$
of~$V_m$ equals
\begin{equation}
\label{eq:eigH}
\diag(im, i(m-2), i(m-4),\ldots,-i(m-2),-im)
\end{equation}
In particular, if $V_m$ is quaternionic
(i.e., $m$ is odd), then each of the eigenvalues
$m^2, (m-2)^2, \ldots, 9, 1$ of
$D_{V_m}(H^2)=-((\rho_{V_m})_*(H))^2$
is of multipicity exactly two, so $c_{V_m}(H^2)\ne0$.

It remains to show condition~(b) of
Proposition~\ref{prop:abc} for $V_m$ with $m$ even.
The case $m=0$ is trivial. Now let $m>0$.
This time it does not suffice to consider
$D_{V_m}(H^2)$ because still all of its nonzero eigenvalues
have multiplicity two, so $b_{V_m}(H^2)=0$.
However, we will show that there exists $\eps>0$
such that $b_{V_m}(H^2+\eps A^2)\ne0$.
Note that
\begin{equation*}
\begin{split}
(\rho_{V_m})_*(A)&: v_{m,\ell}\mapsto i(m-\ell)v_{m,\ell+1}
+i\ell v_{m,\ell-1}\text{ and}\\
-((\rho_{V_m})_*(A))^2&: v_{m,\ell}\mapsto
(m-\ell)(m-\ell-1)v_{m,\ell+2}+\ell(\ell-1)v_{m,\ell-2}
+d_{m,\ell}v_{m,\ell},
\end{split}
\end{equation*}
for some $d_{m,\ell}\in\R$,
where we set $v_{m,k}:=0$ for $k<0$ or $k>m$.
Therefore,
$$
W_0:=\spann\{v_{m,0},v_{m,2},\ldots,v_{m,m}\}\text{ and }
W_1:=\spann\{v_{m,1},v_{m,3},\ldots,v_{m,m-1}\}
$$
are invariant under $-((\rho_{V_m})_*(A))^2$. With
respect to the given basis, the matrix corresponding to
$-((\rho_{V_m})_*(A))^2\restr{W_0}$ is tridiagonal
with subdiagonal entries
$m(m-1), (m-2)(m-3),\ldots, 2\cdot 1$ (and superdiagonal
entries $2\cdot 1,4\cdot 3,\ldots,m(m-1)$). Similarly,
the matrix corresponding to
$-((\rho_{V_m})_*(A))^2\restr{W_1}$ is tridiagonal
with subdiagonal entries
$(m-1)(m-2),(m-3)(m-4),\ldots,3\cdot 2$.
Let $\eps>0$ be arbitrary.
Since the matrix corresponding to
$-((\rho_{V_m})_*(H))^2$ is diagonal, the matrices
corresponding to 
\begin{equation*}
\begin{split}
\Phi_0(\eps):=(D_{V_m}(H^2+\eps A^2))\restr{W_0}\text{ and }
\Phi_1(\eps):=(D_{V_m}(H^2+\eps A^2))\restr{W_1}
\end{split}
\end{equation*}
still are tridiagonal 
with all subdiagonal entries nonzero
(or just an $(1\times 1)$-matrix
in the case of $\Phi_1(\eps)$ for $m=2$).
It is well-known that all eigenvalues
of such maps are of geometric multiplicity one (consider
the rank of $\Phi_k(\eps)-\mu\Id$). In our case, geometric
and algebraic multipicity coincide by~\ref{notrems:abc}(iii).
So for any $\eps>0$,
each of $\Phi_0(\eps)$ and $\Phi_1(\eps)$ has only simple 
eigenvalues.
Also note that $\Phi_0(0)$ and $\Phi_1(0)$ have no common
eigenvalues (recall~\eqref{eq:eigH}).
This implies that there exists $\eps>0$ such that
$D_{V_m}(H^2+\eps A^2)$ has only simple eigenvalues;
in particular, $b_{V_m}(H^2+\eps A^2)\ne0$.
\end{proof}

In the following, let $G$, $\g$ be as in the previous section,
let $G'$ be another compact Lie group, and let $\g'$ be its Lie algebra.

\begin{remark}
\label{rem:prod}
For the direct product $G\times G'$ one has
$$\Irr(G\times G',\C)=\{V\otimes V'\mid V\in\Irr(G,\C), V'\in\Irr(G',\C)\},
$$
where $G\times G'$ acts on $V\otimes V'$ by $\rho_{V\otimes V'}(x,x')=
\rho_V(x)\otimes\rho_{V'}(x')$ (see, e.g., \cite{BtD}, section~II.4). Note that $(V\otimes V')^*\cong V^*\otimes(V')^*$
is isomorphic to $V\otimes V'$ if and only if $V\cong V^*$ and $V'\cong(V')^*$.
In this case, if $J:V\to V$ and $J':V'\to V'$ are as in~\ref{not:regular}(iv),
then $J\otimes J':V\otimes V'\to V\otimes V'$ is a well-defined conjugate linear $G\times G'$-map,
with the sign of its square depending on the signs of $J^2=\pm\Id$ and $(J')^2=\pm\Id$. So one has:
\begin{itemize}
\item[(i)]~\hbox{$V\otimes V'\in\Irr(G\times G',\C)_\R$ $\Longleftrightarrow$} $V$ and $V'$ are both
of real or both of quaternionic type.
\item[(ii)]~\hbox{$V\otimes V'\in\Irr(G\times G',\C)_\C$ $\Longleftrightarrow$} at least one of $V$ or $V'$ is of complex type.
\item[(iii)]~\hbox{$V\otimes V'\in\Irr(G\times G',\C)_\H$ $\Longleftrightarrow$} either $V$ is of real and $V'$ is
of quaternionic type, or vice versa.
\end{itemize}
\end{remark}

\begin{remark}
\label{rem:prodact}
The spaces $\symt(\g)$ and $\symt(\g')$ are canonically embedded in $\symt(\g\oplus\g')$
by the linear extensions of $\iota:Y\cdot Z\mapsto(Y,0)\cdot(Z,0)$ and $\iota':Y'\cdot Z'\mapsto(0,Y')\cdot(0,Z')$,
respectively.
Let $V\in\Irr(G,\C)$ and $V'\in\Irr(G',\C)$. For $(Y,Y')\in\g\oplus\g'$,
$$(\rho_{V\otimes V'})_*(Y,Y')=(\rho_V)_*(Y)\otimes\Id+\Id\otimes(\rho_{V'})_*(Y').
$$
In particular, we have
$$D_{V\otimes V'}(\iota(s))=D_V(s)\otimes\Id\text{ and }D_{V\otimes V'}(\iota'(s'))=\Id\otimes D_{V'}(s')
$$
for $s\in\symt(\g)$, $s'\in\symt(\g')$.
\end{remark}

\begin{lemma}
\label{lem:prod}
Let $V,W\in\Irr(G,\C)$ and $V',W'\in\Irr(G',\C)$.
\begin{itemize}
\item[(i)]
If $a_{V,W}\ne0$ or $a_{V',W'}\ne0$ then $a_{V\otimes V',W\otimes W'}\ne0$.
\item[(ii)]
If $b_V\ne0$ and $b_{V'}\ne0$ then $b_{V\otimes V'}\ne0$.
\item[(iii)]
If $V\in\Irr(G,\C)_\R$ with $b_V\ne0$ and $V'\in\Irr(G',\C)_\H$ with $c_{V'}\ne0$
(or vice versa), then $c_{V\otimes V'}\ne0$.
\end{itemize}
\end{lemma}

\begin{proof}
(i)
If $a_{V,W}\ne0$ choose $s\in\symt(\g)$ with $a_{V,W}(s)\ne0$; that is,
$D_V(s)$ and $D_W(s)$ have no common eigenvalues. By Remark~\ref{rem:prodact}
it follows that
$D_{V\otimes V'}(\iota(s))$ and $D_{W\otimes W'}(\iota(s))$ have no common eigenvalues either.
Thus, $a_{V\otimes V',W\otimes W'}(\iota(s))\ne0$. The case $a_{V',W'}\ne0$ is analogous.

(ii)
Write $d=\dimm V$, $d'=\dimm V'$.
Choose $s\in\symt(\g)$ and $s'\in\symt(\g')$ such that
$b_V(s)\ne 0$ and $b_V(s')\ne0$. Then $D_V(s)$ and $D_{V'}(s')$ both have
only simple eigenvalues, say, $\mu_1<\mu_2<\ldots<\mu_d$ and $\nu_1<\ldots <\nu_{d'}$\,,
respectively. For any $\eps\in\R$, we have
$$D_{V\otimes V'}(\iota(s)+\eps\iota'(s'))=D_V(s)\otimes\Id+\Id\otimes \eps D_{V'}(s')
$$
by Remark~\ref{rem:prodact}, and this operator has eigenvalues
$\mu_i+\eps\nu_j$, $i=1,\ldots,d$, $j=1,\ldots,d'$.
For sufficiently small $\eps>0$, these eigenvalues are again pairwise different, hence
$b_{V\otimes V'}(\iota(s)+\eps\iota'(s'))\ne0$.

(iii)
Here $V\otimes V'$ is of quaternionic type by Remark~\ref{rem:prod}.
Similarly as above, we choose
$s\in\symt(\g)$
and $s'\in\symt(\g')$ such that all eigenvalues of $D_V(s)$ are simple and all eigenvalues
of $D_{V'}(s')$ are of multiplicity exactly two. Then again, choosing $\eps>0$ small enough,
all eigenvalues of
$D_{V\otimes V'}(\iota(s)+\eps\iota'(s'))=D_V(s)\otimes\Id+\Id\otimes \eps D_{V'}(s')$ will have
multiplicity exactly two; hence $c_{V\otimes V'}(\iota(s)+\eps(\iota'(s'))\ne0$.
\end{proof}

\begin{proposition}
\label{prop:prod}
Assume that both $G$ and $G'$ satisfy the conditions (a), (b),~(c) of Proposition~\ref{prop:abc}.
Then $G\times G'$ satisfies condition~(c) of Proposition~\ref{prop:abc}; moreover:
\begin{itemize}
\item[(i)]
$G\times G'$ satisfies condition~(a) of Proposition~\ref{prop:abc} if and only if
$a_{V\otimes V',V^*\otimes V'}\ne0$ and $a_{V\otimes V',V\otimes{V'}^*}\ne0$
whenever $V\in\Irr(G,\C)$, $V'\in\Irr(G,\C)$ are both of complex type.
\item[(ii)]
$G\times G'$ satisfies condition~(b) of Proposition~\ref{prop:abc} if and only if
$b_{V\otimes V'}\ne0$ whenever $V\in\Irr(G,\C)$ and $V'\in\Irr(G',\C)$ are both
of quaternionic type, or if one is of quaternionic and the other of complex type.
\end{itemize}
\end{proposition}

\begin{proof}
In the following, let $V,W\in\Irr(G,\C)$, $V',W'\in\Irr(G',\C)$.

Condition~(c) for $V\otimes V'$ of quaternionic type
follows from Remark~\ref{rem:prod} and Lemma~\ref{lem:prod}(iii).

The condition of (i) is necessary for condition~(a) of Proposition~\ref{prop:abc} because
for $V,V'$ of complex type, $V\otimes V'\not\cong V^*\otimes V', V\otimes(V')^*$ and
$(V\otimes V')^*\cong V^*\otimes(V')^*\not\cong V^*\otimes V', V\otimes(V')^*$.
For the converse direction, assume that the condition of~(i) is satisfied; we have to show
that this already implies condition~(a) of Proposition~\ref{prop:abc} for $G\times G'$.
Note that if $V\otimes V'\not\cong W\otimes W'$ and $(V\otimes V')^*\not\cong W\otimes W'$
then one of the following three conditions holds:
\begin{itemize}
\item[1.]
$V,V^*\not\cong W$ or $V',(V')^*\not\cong W'$,
\item[2.] 
$V\cong W$ and $(V')^*\cong W'$, but $V^*\not\cong W$ and $V'\not\cong W'$,
\item[3.]
$V^*\cong W$ and $V'\cong W'$, but $V\not\cong W$ and $(V')^*\not\cong W'$.
\end{itemize}
Since $G$ and $G'$ satisfy the conditions of Proposition~\ref{prop:abc} by
assumption, case~1.~implies that $a_{V,W}\ne0$ or $a_{V',W'}\ne0$.
By Lemma~\ref{lem:prod}(i) we then have $a_{V\otimes V',W\otimes W'}\ne0$.
In case~2., $V$ and $V'$ are both nonisomorphic to their duals and hence are of
complex type. Moreover, $W\otimes W'\cong V\times(V')^*$,
so $a_{V\otimes V',W\otimes W'}=a_{V\otimes V',V\otimes (V')^*}\ne0$ by assumption.
Case~3.~is analogous.
Thus, condition~(a) of Proposition~\ref{prop:abc} is satisfied for $G\times G'$.

The condition of~(ii) is necessary for condition~(b) of Proposition~\ref{prop:abc} by
Remark~\ref{rem:prod}(i), (ii).
For the converse direction, assume that the condition of~(ii) is satisfied; that is,
$b_{V\otimes V'}\ne0$ whenever both of $V$ and $V'$ are of quaternionic type, or
if one is of quaternionic and one is of complex type.
By Lemma~\ref{lem:prod}(ii) we know $b_{V\otimes V'}\ne0$ if none of $V$ or $V'$ is of quaternionic type.
By Remark~\ref{rem:prod}, these were all possible cases for $V\otimes V'$ of real or complex
type. So condition~(b) of Proposition~\ref{prop:abc} holds for $G\times G'$.
\end{proof}

\begin{remark}
\label{rem:nontriv}
Note that the conditions in (i),~(ii) of Proposition~\ref{prop:prod}
are far from trivial, in spite of the assumption that $G$ and $G'$ individually satisfy
the conditions of Proposition~\ref{prop:abc} separately. For example,
if $V$ and~$V'$ are both of quaternionic type, then for generic $s,s'$, all eigenvalues
of $D_V(s)$ and $D_{V'}(s')$ will be of multiplicity exactly two, which results in
$D_{V\otimes V'}(\iota(s)+\iota'(s'))$ having all of its eigenvalues of multiplicity exactly four.
But $V\otimes V'$ is of real type, so condition~(b) of Proposition~\ref{prop:abc} requires
generic $D_{V\otimes V'}(\tilde s)$ to have all eigenvalues simple. Thus, in order to establish
this condition, it will not suffice to work with elements of the form $\iota(s)+\iota'(s')
\in\symt(\g\oplus\g')$. We will succeed in solving this problem in the case $G=G'=\SU(2)$;
see Lemma~\ref{lem:pairs}(ii) below.
\end{remark}

We now state our main result:

\begin{theorem}\
\label{thm:spinmult}
Any product $\SU(2)\times\ldots\times\SU(2)$ or $\SU(2)\times\ldots\SU(2)\times T^n$, where $T^n$ is
a torus,
satisfies the conditions of Proposition~\ref{prop:abc}.
Consequently, for a generic left invariant metric~$g$ on $\SU(2)\times\ldots\times\SU(2)$
or on $\SU(2)\times\ldots\times\SU(2)\times T^n$, the associated Laplace
operator~$\Delta_g$ has irreducible real eigenspaces.
In particular, this is the case for $\SU(2)\times\SU(2)=\Spin(3)\times\Spin(3)=\Spin(4)$.
\end{theorem}

The following Lemma will be the key to the proof of Theorem~\ref{thm:spinmult}.
We continue to use the notation from Example~\ref{ex:torus} and from the
proof of Theorem~\ref{thm:su2} concerning the irreducible representations
of $T^n$ and $\SU(2)$, respectively. Recall that all nontirival irreducible representations~$V_\lambda$ of~$T^n$
are $1$-dimensional and of complex type, and that the $(m+1)$-dimensional representation~$V_m$
of $\SU(2)$ is of quaternionic type if $m$ is odd, and of real type otherwise.

\begin{lemma}\
\label{lem:pairs}
\begin{itemize}
\item[(i)]
If $m\in\N$ is odd and $0\ne\lambda\in(\Z^n)^*$
then $V_m\otimes V_\lambda\in\Irr(\SU(2)\times T^n,\C)_\C$
satisfies $b_{V_m\otimes V_{\lambda}}\ne0$.
\item[(ii)]
If $m,m'\in\N$ are odd, then $V_m\times V_{m'}\in\Irr(\SU(2)\times\SU(2),\C)_\R$ satisfies
$b_{V_m\otimes V_{m'}}\ne0$.
\end{itemize}
\end{lemma}

\begin{proof}
(i)
Let $H\in\su(2)=\g'$ be as in the proof of Theorem~\ref{thm:su2}, and choose
$Y\in\R^n=\g$ such that $\lambda(Y)\ne0$.
Let
$$s:=(H,0)\cdot(0,Y)\in\symt(\g\oplus\g').
$$
Write $V:=V_m$, $V':=V_\lambda$\,.
Note that $(\rho_{V\lambda\otimes V'})_*(H,0)=(\rho_V)_*(H)\otimes\Id$ and
$(\rho_{V\otimes V'})_*(0,Y)=\Id\otimes(\rho_{V'})_*(Y)$
commute. Thus,
\begin{equation*}
\begin{split}
D_{V\otimes V'}(s)&=-((\rho_V)_*(H)\otimes\Id)\circ(\Id\otimes(\rho_{V'})_*(Y))
=-(\rho_V)_*(H)\otimes(\rho_{V'})_*(Y)\\&=-(\rho_{V_m})_*(H)\otimes 2\pi i\lambda(Y)\Id.
\end{split}
\end{equation*}
Recall that the eigenvalues $im,i(m-2),\ldots,-i(m-2),-im$
of $(\rho_{V_m})_*(H)$ are all simple. 
Since $V=V_\lambda$ has dimension~$1$,
the eigenvalues $2\pi m\lambda(Y),2\pi(m-2)\lambda(Y),\ldots, -2\pi m\lambda(Y)$
of $D_{V\otimes V'}(s)$ are all simple, too. In particular, $b_{V\otimes V'}(s)\ne0$. This
shows $b_{V\otimes V'}\ne0$, as desired.

(ii)
Write $V:=V_m$ and $V':=V_{m'}$\,.
We continue to use the basis $\{H,A,B\}$ of $\g=\g'=\su(2)$ from
the proof of Theorem~\ref{thm:su2}.
Since $H$ and $B$ are conjugate matrices in $\su(2)$, they
are conjugate by an element of $\SU(2)$. So $(\rho_V)_*(B)$
has the same eigenvalues $im,\ldots,-im$ as $(\rho_V)_*(H)$,
and similarly for $V'$\,.
Choose some $\eps\in(0,\frac1{m'})$ and consider the auxiliary elements
$$s_H:=(H,\eps H)^2\text{ and }s_B:=(B,\eps B)^2\in\symt(\g\oplus\g')=\symt(\g\oplus\g).
$$
Each of the operators
\begin{equation*}
\begin{split}
\phi&:=(\rho_{V\otimes V'})_*(H,\eps H)=(\rho_V)_*(H)\otimes\Id+\Id\otimes(\rho_{V'})_*(\eps H)
   \text{ and }\\
\psi&:=(\rho_{V\otimes V'})_*(B,\eps B)=(\rho_V)_*(B)\otimes\Id+\Id\otimes(\rho_{V'})_*(\eps B)
\end{split}
\end{equation*}
has the eigenvalues $\pm(ik\pm\eps ik')$, $k\in\{1,3,\ldots,m\}$, $k'\in\{1,3,\ldots,m'\}$.
Due to the choice of~$\eps$, all of these eigenvalues are simple (and nonzero).
So each of the operators
$$D_{V\otimes V'}(s_H)=-\phi^2\text{ and }D_{V\otimes V'}(s_B)=-\psi^2
$$
has the eigenvalues $(k\pm\eps k')^2$, all positive and of multiplicity exactly two.
Although multiplicity two is already better than multiplicity four (recall the considerations
in Remark~\ref{rem:nontriv}),
showing that the multiplicities become simple for generic~$s$ requires a little more work.
For $\alpha\in\R$, let
$$D_\alpha:=D_{V\otimes V'}((1-\alpha) s_H+\alpha s_B)=-((1-\alpha)\phi^2+\alpha\psi^2).
$$
We are going to show, specifically, that for all~$\alpha$ in some dense open set
$\mathcal{O}\subset\R$, $D_\alpha$ has only
simple eigenvalues.

Let $x:=\exp(\frac\pi 2 B)=\left(\begin{smallmatrix}0&-1\\1&0\end{smallmatrix}\right)$.
Note that this is the same matrix as~$B$, but this time regarded as an element of $\SU(2)$,
not its Lie algebra.
Let
$$T:=\rho_{V\otimes V'}(x,x)=\rho_V(x)\otimes\rho_{V'}(x).
$$
Recalling the definition of the basis elements $v_{m,\ell}$ of $V=V_m$ and the
definition of the action $\rho_V$\,, note the following facts:
\begin{itemize}
\item[1.)]
$T$ is an involution; i.e., $T^2=\Id$. In fact, note that $x^2=\left(\begin{smallmatrix}-1&0\\0&-1
\end{smallmatrix}\right)\in\SU(2)$ acts on both $V$ and~$V'$ as $-\Id$. Thus,
$T^2=\rho_V(x^2)\otimes\rho_{V'}(x^2)=-\Id\otimes(-\Id)=\Id$.
\item[2.)]
$T$ preserves the real span
$$\mathcal{R}:=\spann_\R\{v_{m,\ell}\otimes v_{m',\ell'}\mid 0\le\ell\le m,0\le\ell'\le m'\}
$$
of the basis elements $v_{m,\ell}\otimes v_{m',\ell'}$ of~$V\otimes V'$. This follows from the fact
that the matrix $x$ has real entries, and from the definition of $\rho_V$\,.
\item[3.)]
$T$ anticommutes with~$\phi$. In fact, $\operatorname{Ad}_x(H)=xHx\inv=-H$, hence
$\rho_V(x)\circ(\rho_V)_*(H)\circ\rho_V(x)\inv=-(\rho_V)_*(H)$, and similarly for~$V'$.
\item[4.)]
$T$ commutes with~$\psi$. This is obvious from the definitions, since $x=\exp(\frac\pi2 B)$.
\end{itemize}
Let $W^+,W^-\subset V\otimes V'$ denote the $1$-, resp.~$(-1)$-eigenspace of the involution~$T$.
Both are invariant under $\phi^2$ and $\psi^2$ by 3.)~and~4.), hence under each of the maps~$D_\alpha$\,.

Since $\phi$ anticommutes with~$T$, it interchanges
$W^+$ and~$W^-$. Moreover, $\phi$ is invertible and
preserves eigenspaces of~$\phi^2$. Since all eigenvalues
of~$\phi^2$ were of multiplicity two, this implies that
$D_0\restr{W^+}=-\phi^2\restr{W^+}$ has only simple eigenvalues, and so does $D_0\restr{W^-}$.
It follows that there is a dense open set $\mathcal{O}_0\subset\R$ such that $D_\alpha\restr{W^+}$
and $D_\alpha\restr{W^-}$ both have only simple eigenvalues.

On the other hand, $\psi$ commutes with~$T$, so $\psi$ preserves both $W^+$ and $W^-$.
By~2.)~above, $W^+$ and $W^-$ are spanned by their intersections with the real vector space~$\mathcal{R}$.
Also note that $\psi$ leaves~$\mathcal{R}$ invariant, being the initial derivative
of the family of operators $\rho_{V\otimes V'}(t(B,\eps B))$ which clearly preserve~$\mathcal{R}$.
Since the eigenvalues of~$\psi$ are purely imaginary and nonzero,
it follows that the eigenvalues of~$\psi$ on $W^+\cap\mathcal{R}$
come in conjugate pairs; therefore, all eigenvalues of $\psi^2\restr{W^+}$ have multiplicity at least two.
Analogously, the same holds for $\psi^2\restr{W^-}$\,. However, we already saw above
that all eigenvalues of~$\psi^2$ are of multiplicity exactly two. Therefore,
$D_1\restr{W^+}=-\psi^2\restr{W^+}$ and $D_1\restr{W^-}=-\psi^2\restr{W^-}$
can have no eigenvalues in common.
It follows that there is a dense open set $\mathcal{O}_1\subset\R$ such that $D_\alpha\restr{W^+}$
and $D_\alpha\restr{W^-}$ have no eigenvalues in common.

Consequently, for all $\alpha\in\mathcal{O}:=\mathcal{O}_0\cap\mathcal{O}_1$ the operator $D_\alpha$
has both of the above properties, and can therefore have only simple eigenvalues. So
$b_{V\otimes V'}(\alpha s_H+(1-\alpha)s_B)\ne0$ for these~$\alpha$, which shows $b_{V\otimes V'}\ne0$,
as desired.
\end{proof}

\begin{proof}[Proof of Theorem~\ref{thm:spinmult}.]\
We can treat both types of products simultaneously by admitting $n=0$, in which case the torus
is the trivial group~$\{e\}$ (possessing only the trivial irreducible representation~$V_0$).

Since $\SU(2)$ has no irreducible representations of complex type,
$\SU(2)\times\ldots\times\SU(2)$ has no such representations either, by Remark~\ref{rem:prod}. Using
Theorem~\ref{thm:su2} and Example~\ref{ex:torus}, and applying
Proposition~\ref{prop:prod}(i) repeatedly, we conclude that $G:=\SU(2)\times\ldots\times\SU(2)\times T^n$
satisfies condition~(a) of Proposition~\ref{prop:abc}.
It remains to show conditions (b) and~(c).
Let $k$ be the number of factors equal to~$\SU(2)$ in~$G$.
Using the same notation as before
for the irreducible representations of $\SU(2)$ and~$T^n$, let
$$V:=V_{m_1}\otimes\ldots\otimes V_{m_k}\otimes V_\lambda
$$
be an arbitrary irreducible representation of~$G$.
Let $\ell$ denote the number of factors $V_{m_j}$ of quaternionic type.
Whether or not $V$ satisfies $b_V\ne0$, resp.~$c_V\ne0$,
does obviously not depend on the ordering of the first $k$ factors. We order
them such that the product starts with $\ell':=2\lfloor\ell/2\rfloor$ factors of quaternionic type,
continues with factors of real type, and
$V_{m_k}$ is of either real or quaternionic type, depending on whether~$\ell$ is even or odd.
By Lemma~\ref{lem:pairs}(ii),
$b_{V_{m_1}\otimes V_{m_2}}\ne0$, \dots, $b_{V_{m_{\ell'-1}}\otimes V_{m_{\ell'}}}\ne0$.
We also have $b_{V_{m_j}}\ne0$ for the $V_{m_j}$ of real type (see Theorem~\ref{thm:su2}), and $b_{V_\lambda}\ne0$
(see Example~\ref{ex:torus}).
In the case that $\ell$ is even, one immediately concludes $b_V\ne0$ by using Lemma~\ref{lem:prod}(ii)
repeatedly.

Now let $\ell$ be odd. We still have $b_{V_{m_1}\otimes\ldots\otimes V_{m_{k-1}}}\ne0$, and $V_{m_k}$ is
of quaternionic type. If $\lambda\ne0$ then $b_{V_{m_k}\otimes V_\lambda}\ne0$ by Lemma~\ref{lem:pairs}(i).
By Lemma~\ref{lem:prod}(ii) we again obtain $b_V\ne0$.
If $\lambda=0$, then $V_\lambda$ is the trivial represenation (of real type),
hence $V$ is of quaternionic type.
Using $c_{V_{m_k}}\ne0$ (see Theorem~\ref{thm:su2})
and applying Lemma~\ref{lem:prod}(iii) twice, we obtain $c_V\ne0$.
\end{proof}

\begin{lemma}
\label{lem:quot}
Let $G$ be a compact Lie group satisfying the conditions of Proposition~\ref{prop:abc},
and let $\Gamma$ be some discrete central subgroup of~$G$. Then $\bar G:=G/\Gamma$
again satisfies the conditions of Proposition~\ref{prop:abc}.
\end{lemma}

\begin{proof}
Note that $\Irr(\bar G,\C)$ can be considered as a subset of $\Irr(G,\C)$,
consisting of precisely those irreducible representations of~$G$ which restrict to the
trivial representation on~$\Gamma$. It is easy to see that this inclusion respects the different
types of irreducible representations (real, complex, quaternionic) from~\ref{not:regular}.
Moreover, the Lie algebras of $G$ and~$\bar G$ coincide.
Therefore, the conditions of Proposition~\ref{prop:abc} for~$\bar G$ amount to just a certain subset
of the conditions for~$G$, which are satisfied by assumption.
\end{proof}

\begin{corollary}
\label{cor:quot}
Any quotient of a Lie group of the form $\SU(2)\times\ldots\times\SU(2)\times T^n$ by a discrete
central subgroup satisfies the conditions of
Proposition~\ref{prop:abc}.
In particular, for a generic left invariant metric~$g$ on any such quotient, the associated Laplace
operator~$\Delta_g$ has irreducible real eigenspaces. For example, this holds for the
compact Lie groups $\SO(3)$, $\U(2)$, and $\SO(4)$.
\end{corollary}

\begin{proof}
This follows immediately from Theorem~\ref{thm:spinmult} and Lemma~\ref{lem:quot}; also note that
$\SO(3)\cong\SU(2)/\{\pm\Id\}$, $\SO(4)\cong(\SU(2)\times\SU(2))/\{\pm(\Id,\Id)\}$, and
$\U(2)\cong (\SU(2)\times S^1)/\{\pm(\Id,1)\}$, where
$S^1=T^1$ is considered as $S^1\subset\C$.
\end{proof}

\end{document}